\documentclass{article}

\usepackage{RegretLBMPC}

\usepackage[utf8]{inputenc} 
\usepackage[T1]{fontenc}   
\usepackage[pdfencoding=unicode, psdextra]{hyperref}      
\usepackage{url}            
\usepackage{booktabs}      
\usepackage{nicefrac}       
\usepackage{microtype}      
\usepackage{graphicx}
\usepackage{doi}
\usepackage{authblk}
\usepackage{cite}
\usepackage{amsmath,amssymb,amsfonts}
\usepackage{algorithm}
\usepackage{algpseudocode}
\newtheorem{theorem}{Theorem}
\newtheorem{defn}{Definition}
\newtheorem{assm}{Assumption}
\newtheorem{remark}{Remark}
\newtheorem{prop}{Proposition}

\newtheorem{lem}{Lemma}

\usepackage{svg}
\usepackage{subcaption}
\usepackage[pdfencoding=unicode, psdextra]{hyperref}
\usepackage{bookmark}
\usepackage{float, csquotes}

\DeclareMathOperator*{\argmin}{arg\,min}
\usepackage{cuted}
\usepackage{breqn}

\usepackage{enumitem}
\usepackage{multicol}
\usepackage{adjustbox}
\usepackage{titlesec}
\titlelabel{\thetitle.\quad}

\title{Regret Analysis of Learning-Based MPC with Partially-Unknown Cost Function}

\hypersetup{
pdftitle={Regret Analysis of Learning-Based MPC with Partially-Unknown Cost Function},
pdfsubject={math.OC, cs.LG},
pdfauthor={Ilgin Dogan, Zuo-Jun Max Shen, Anil Aswani},
}

\begin{document}

\author[1]{\large Ilgin Dogan}
\author[1, 2]{Zuo-Jun Max Shen}
\author[1]{Anil Aswani\thanks{This material is based upon work supported by the National Science Foundation under Grant CMMI-1847666.} }
\affil[1]{\small Industrial Engineering and Operations Research, University of California, Berkeley 94720, \{ilgindogan,maxshen,aaswani\}@berkeley.edu}
\affil[2]{\small Faculty of Engineering and Faculty of Business and Economics, University of Hong Kong, China, maxshen@hku.hk}

\maketitle

\begin{abstract}
The exploration/exploitation trade-off is an inherent challenge in data-driven adaptive control. Though this trade-off has been studied for multi-armed bandits (MAB's) and reinforcement learning for linear systems; it is less well-studied for learning-based control of nonlinear systems. A significant theoretical challenge in the nonlinear setting is that there is no explicit characterization of an optimal controller for a given set of cost and system parameters. We propose the use of a finite-horizon oracle controller with full knowledge of parameters as a reasonable surrogate to optimal controller. This allows us to develop policies in the context of learning-based MPC and MAB's and conduct a control-theoretic analysis using techniques from MPC- and optimization-theory to show these policies achieve low regret with respect to this finite-horizon oracle. Our simulations exhibit the low regret of our policy on a heating, ventilation, and air-conditioning model with partially-unknown cost function. \\
	
\end{abstract}

\noindent\makebox[\linewidth]{\rule{\textwidth}{1pt}} 

\section{Introduction}
\label{sec:introduction}
Reinforcement learning (RL) research \cite{AbbasiYadkori2019ModelFreeLQ, Chen2019ReinforcementLA, Agarwal2020BoostingFC} focuses on regret analysis for primarily unconstrained, linear systems. On the other hand, adaptive model predictive control (MPC), including learning-based MPC (LBMPC), seeks to ensure constraint satisfaction in the presence of models that are updated as more data becomes available \cite{Negenborn2004ExperiencebasedMP, aswani2013provably, Karnchanachari2020PracticalRL, Gros2020ReinforcementLF}.  The relationship between MPC and RL has not yet been fully explored. 

Our paper aims to better connect these two areas. We make two main contributions: First, we discuss how comparing finite-horizon policies with different horizon lengths leads to ambiguous regret notions in evaluation of learning-based control policies. Thus we propose a regret notion that compares a finite-horizon learning-based policy with a finite-horizon oracle controller as the benchmark. Second, we bound this regret notion for a class of learning-based control policies for which we prove constraint satisfaction. An important aspect of our regret analysis is that we have to consider the stability of our policy when bounding the regret. In this sense, our analysis draws a connection between stability of the nonlinear control system and regret performance of the learning policy.
\subsection{Partially-Unknown Cost Function} \label{sec:motivation}
MPC usually assumes the system dynamics and a cost function are exactly known. However, these may be partially-unknown in real-world systems that motivate our setup.
\subsubsection{Heating, Ventilation, Air-Conditioning (HVAC) Systems}\label{subsec:hvac}
Since HVAC uses a large part of total building energy, improving HVAC energy-efficiency using MPC has been studied \cite{aswani2011reducing, afram2014theory, ostadijafari2019linear, fang2020identification}. However, past works typically assume perfect knowledge of a cost function that characterizes the trade-off between energy-efficiency and occupant comfort. In practice, the quantity of trade-off is different for each occupant and is \emph{a priori} unknown to the controller. It makes sense to learn an ideal trade-off from occupant-reported data \cite{aswani2018inverse} and then adapt the MPC operation in response, which is an example of MPC with a partially-unknown cost function. 
\subsubsection{Clinical-Inventory Management}
Inventory management in hospitals involves periodically restocking drugs and medical supplies, and MPC for inventory management \cite{velarde2014application, schildbach2016scenario, maestre2018application, garcia2020data} is powerful as it naturally captures the dynamics of consuming and purchasing drugs and supplies. Although past work typically assumes that consumption dynamics are completely characterized, it is not realistic for the demand in hospitals due to unforeseeable medical emergencies. It then makes sense from a practical standpoint to learn about the demand from such events and then adapt the MPC operation in response, which is an example of MPC with learning for the dynamics.
\subsection{Exploration/Exploitation Trade-Off}
A challenge in LBMPC is to jointly optimize the control to minimize a cost function and to steer the system to get more information about unknown system or cost parameters \cite{Mesbah2018StochasticMP}. This \textit{exploration}/\textit{exploitation} trade-off and has been formally studied in the setting of MAB's \cite{thompson1933likelihood,agrawal2013thompson,mintz2017non}, RL for finite Markov chains \cite{heger1994consideration, biyik2019efficient, budd2020markov} and for linear systems \cite{735224, KiumarsiKhomartash2017HCO, pmlr-v97-cohen19b, pmlr-v119-simchowitz20a}.
Most work on MAB's assumes (weak-)stationarity because computing the optimal policy with non-stationary is PSPACE-hard \cite{papadimitriou1999complexity}. In RL of control systems, past work on nonlinear systems is limited \cite{koller2018learning, gros2019data, kakade2020information, wabersich2020performance, fan2020, boffi2021regret} because the optimal controller for linear systems with a quadratic cost is completely characterized by the Algebraic Ricatti Equation: This allows one to convert the RL problem into simply a parameter estimation problem. However, extending these ideas to nonlinear systems is nontrivial as there is no such simple characterization of the optimal controller, and so alternative approaches are needed. We design a learning-based controller for nonlinear and non-stationary systems where the policy explores to improve the estimation methodology embedded in the learning mechanism.
\subsection{Outline}
Sect. \ref{sec:prelim} covers preliminaries. Sect. \ref{sec:problemform} defines our setup and and proves safety properties for a class of control policies. Sect. \ref{sec:Nstepregret} introduces \textit{$N$-step dynamic regret}, and Sect. \ref{sec:estimate} and \ref{sec:approach} present a finite sample analysis for the parameter estimation and regret analysis for the \textit{non-myopic $\epsilon$-greedy algorithm}. Lastly, numerical experiments are done in Sect. \ref{sec:exper}.
\section{Preliminaries} \label{sec:prelim}
A polytope $\mathcal{U}$ in $\mathbb{R}^n$ can be represented as intersection of a set of half-spaces \cite{borelli2009constrained}: $\mathcal{U} = \{x: P_i x \leq q_i, i = 1, \ldots, d\}$, $P_i \in \mathbb{R}^{d \times n}$, $q_i \in \mathbb{R}^d$. Let $\mathcal{U}, \mathcal{V}$ be two sets. The linear transformation of $\mathcal{U}$ by a matrix $\mathcal{R}$ is $\mathcal{R}\mathcal{U} = \{\mathcal{R}u:u \in \mathcal{U}\}$. Their Minkowski sum \cite{Schneider1993ConvexBT} is defined as $\mathcal{U} \oplus \mathcal{V} = \{u + v : u \in \mathcal{U}; v \in \mathcal{V}\}$ and Pontryagin set difference \cite{Kolmanovsky1998TheoryAC} is defined as $\mathcal{U} \ominus \mathcal{V} = \{u: u + \mathcal{V} \subseteq \mathcal{U}\}$. Note $\mathcal{R}(\mathcal{U} \ominus \mathcal{V}) \subseteq \mathcal{R}\mathcal{U} \ominus \mathcal{R}\mathcal{V}$ and $(\mathcal{U} \ominus \mathcal{V}) \oplus \mathcal{V} \subseteq \mathcal{U}$.
\section{Problem Formulation}\label{sec:problemform}

Let $x_t\in\mathbb{R}^n$ be states and $u_t\in\mathbb{R}^q$ be inputs. We assume $x_t\in \mathcal{X}$ and $u_t\in\mathcal{U}$ are constrained by (compact) polytopes $\mathcal{X},\mathcal{U}$. The true system dynamics are $x_{t+1} = f(x_t, u_t, \theta_0) = Ax_t + Bu_t + g(x_t,u_t,\theta_0)$, where $A\in\mathbb{R}^{n\times n}$, $B\in\mathbb{R}^{n\times q}$, $\theta_0\in\Theta$ for some compact set $\Theta\subseteq\mathbb{R}^p$, and the nonlinear function $g(\cdot,\cdot,\theta) : \mathbb{R}^n \times \mathbb{R}^q \rightarrow \mathbb{R}^n$ is parameterized by $\theta\in\Theta$. We assume $\big\{g(x,u,\theta_0) : x\in\mathcal{X}, u\in\mathcal{U}\big\} \subseteq \mathcal{W}$ for some bounded polytope $\mathcal{W}$, and $A,B,g,\mathcal{W},\Theta$ are known but $\theta_0$ is not known to the controller. Define $w_t = g(x_t,u_t,\theta_0)$, and note $w_t\in\mathcal{W}$ by assumption. The intuition is we have a nominal linear model and a partially-unknown, nonlinear correction.

At each time $t$, the controller receives a stochastic reward $r_t$ from distribution $\mathbb{P}_{x_t, u_t,\theta_0}$ with probability density function $p(r|x_t,u_t,\theta_0)$ and expectation $\mathbb{E}\,r_t = h(x_t,u_t,\theta_0)$. We assume $h$ is parametrically unknown ($\theta_0$ is unknown). This setup can handle stochastic costs $c_t$ (as opposed to rewards) by setting $r_t = -c_t$. We standardize our notation for rewards. 

The control problem is to sequentially choose inputs to maximize expected total reward at the end of a finite time horizon $\mathcal{T}= \{0,\ldots,T\}$. At time $t$, the controller has access to past rewards, inputs, and states. Hence, any policy $u_t = \Lambda_t(\mathcal{F}_t)$ will be a sequence (with respect to $t$) of functions of
\begin{equation}
\label{eqn:ft}
    \mathcal{F}_t = \{r_0,\ldots,r_{t-1}, u_0,\ldots,u_{t-1}, x_0,\ldots,x_t\}.
\end{equation}
We distinguish between different policies by using superscripts for the sequence of functions $\Lambda_t$ characterizing the policy.
\subsection{Learning-Based MPC Formulation}
LBMPC uses two models: a learned model to enhance performance and a nominal model to provide robustness \cite{aswani2013provably}. Because $A,B$ are known in our setup, the controller uses as its nominal model $\bar{x}_{t+k+1|t} = A\bar{x}_{t+k|t} + Bu_{t+k|t}$, where $\bar{x}\in\mathbb{R}^n$ is system state of the nominal model. The ``$|t$'' notation denotes the initial condition is taken to be $\bar{x}_{t|t} = x_t$, where $x_t$ is the true state at time $t$. Because $g(\cdot,\cdot,\theta)$ is also known, the controller uses as its learned model $\tilde{x}_{t+k+1|t} = A\tilde{x}_{t+k|t} + Bu_{t+k|t} + g(\tilde{x}_{t+k|t}, u_{t+k|t}, \widehat{\theta}_t)$, where $\tilde{x}$ is the system state of the learned model and $\widehat{\theta}_t$ is the controller's estimate of $\theta_0$ at time $t$. Here, LBMPC learns the true dynamics by updating its estimate of $\theta_0$ as more state measurements become available.

We must first discuss the terminal set used for the MPC. Assuming that $(A,B)$ is stabilizable, there exists a constant state-feedback matrix $K\in\mathbb{R}^{q\times n}$ such that $(A+BK)$ is Schur stable. We assume $\Omega \subseteq \mathcal{X}$ is a maximal output admissible disturbance invariant set \cite{Kolmanovsky1998TheoryAC} meaning that for some stabilizing $K$ it satisfies: a) $\Omega \subseteq \{\overline{x} : \overline{x} \in \mathcal{X} :  K\overline{x} \in \mathcal{U}\}$ (\emph{constraint satisfaction}) and b) $(A + BK)\Omega \oplus \mathcal{W} \subseteq \Omega$ (\emph{disturbance invariance}). The intuition is that $\Omega$ is a set of states satisfying the constraints $\mathcal{X}$ for which there exists a feasible action keeping the true state within $\Omega$ despite the uncertainty of the nominal model. Several algorithms \cite{Kolmanovsky1998TheoryAC, limon2009input, rakovic2010parameterized, wang2021computation} can compute this set, and so we assume $\Omega$ is available to the controller. 

With the set $\Omega$, we consider an (simplified) LBMPC variant that maximizes the expected $N$-step reward. Our results can be generalized straightforwardly to the full formulation \cite{aswani2013provably}, but we do not consider this as it adds substantial notational complexity that hinders showcasing the stochastic aspects of our setting. The LBMPC formulation of a finite-horizon $N$ is
\begin{align}
    V_N(x_t, \theta, t) = &\max \textstyle\sum_{k=0}^{N} h(\tilde{x}_{t+k|t}, u_{t+k|t}, \theta) \allowdisplaybreaks \notag\\ 
    \mathrm{s.t.}\ & \bar{x}_{t+k+1|t} = A\bar{x}_{t+k|t} + Bu_{t+k|t} && \textstyle  k \in \langle N-1 \rangle \allowdisplaybreaks \notag \\
    & \tilde{x}_{t+k+1|t} = A\tilde{x}_{t+k|t} + Bu_{t+k|t} + g(\tilde{x}_{t+k|t}, u_{t+k|t}, \theta) && \textstyle k \in \langle N-1 \rangle \allowdisplaybreaks \notag \\
    & \bar{x}_{t+k|t} \in \mathcal{X} && \textstyle k \in [N] \allowdisplaybreaks \notag \\
    & u_{t+k|t} \in \mathcal{U} && \textstyle k \in \langle N\rangle \notag \allowdisplaybreaks \\
    & \bar{x}_{t+1|t} \in \Omega \ominus \mathcal{W}, \ \bar{x}_{t|t} = \tilde{x}_{t|t} = x_t  \allowdisplaybreaks \label{eqn:lbmpc}
\end{align}
where $\langle k \rangle = \{0,\ldots,k\}$ and $[k] = \{1,\ldots,k\}$. The difference between this simplified variant and the full formulation is that here we apply the invariant set $\Omega$ at the first time step, an idea previously used in \cite{6315483}, whereas the full formulation uses a robust tube framework to apply $\Omega$ at the $N$-th time point. Our results apply to the above LBMPC formulation and may generalize to the similar variants, but it is unclear if they would generalize to other LBMPC forms without further study.
\subsection{Safety of Learning-Based MPC Variant}

Because applying the invariant set to the first time point in an MPC formulation is nonstandard, we first formally prove that this LBMPC variant ensures recursive properties of robust constraint satisfaction and robust feasibility.

\begin{theorem} \label{thm:safety}
Suppose $\{u_{t|t},\ldots,u_{t+N|t}\}$ are feasible for $V_N(x_t,\theta,t)$ for any $\theta$. If $\Omega$ is a maximal output admissible disturbance invariant set, then choosing $u_t = u_{t|t}$ ensures that we have: a) $x_{t+1} \in \mathcal{X}$ (robust constraint satisfaction) and b) there exist values $\{u_{t+1|t+1},\ldots,u_{t+N|t+1}\}$ that are feasible for $V_N(x_{t+1},\theta',t+1)$ for any $\theta'$ (robust feasibility).
\end{theorem}

\begin{proof}
Since $\{u_{t|t},\ldots,u_{t+N|t}\}$ are feasible for $V_N(x_t,\theta,t)$, then $\overline{x}_{t+1|t} = Ax_t + Bu_{t|t} \in \Omega \ominus \mathcal{W}$ by (\ref{eqn:lbmpc}). By relating the true dynamics to the nominal model, the true next state is $x_{t+1} = \overline{x}_{t+1|t} + w_t$ for some $w_t \in \mathcal{W}$. This means $x_{t+1} \in (\Omega \ominus \mathcal{W}) \oplus \mathcal{W} \subseteq \Omega \subseteq \mathcal{X}$ where the last set inclusion follows from the constraint satisfaction property in the definition of $\Omega$. 

By the definition (\ref{eqn:lbmpc}) of $V_N(x_{t+1},\theta',t+1)$, we have that $\bar{x}_{t+1|t+1} = x_{t+1}$. However, we just showed that $x_{t+1}\in\Omega$. Hence $\bar{x}_{t+1|t+1} \in \Omega$. Now set $u_{t+1|t+1} = Kx_{t+1}$, and note that the constraint satisfaction property of $\Omega$ means $u_{t+1|t+1} \in\mathcal{U}$. Since $\overline{x}_{t+2|t+1} = A\overline{x}_{t+1|t+1} + Bu_{t+1|t+1} = (A + BK)\overline{x}_{t+1|t+1}$, we have $\overline{x}_{t+2|t+1} \in (A+BK)\Omega \subseteq ((A+BK)\Omega \oplus \mathcal{W})\ominus\mathcal{W} \subseteq  \Omega\ominus\mathcal{W}$ where the last set inclusion follows by the disturbance invariance property of $\Omega$. So $\bar{x}_{t+2|t+1} \in \Omega\ominus\mathcal{W} \subseteq\Omega\subseteq \mathcal{X}$ by the constraint satisfaction property of $\Omega$. We can sequentially repeat this argument with $u_{t+k+1|t+1} = Kx_{t+k+1|t+1}$ to show this choice results in $u_{t+k+1|t+1}\in \mathcal{U}$ and $x_{t+k+1|t+1}\in\mathcal{X}$ for $k \in [N-1]$. Thus $\{u_{t+1|t+1},\ldots,u_{t+N|t+1}\}$ are feasible for $V_N(x_{t+1},\theta',t+1)$.
\end{proof}
\begin{remark}
An important feature of the above result is that there is no required relationship between the $\theta$ and $\theta'$. Since estimates of the $\theta$ are updated through learning, this shows that the safety properties of this LBMPC variant are decoupled from the design of the learning-process. 
\end{remark}
\subsection{Technical Assumptions}
Our learning-based control problem is well-posed under certain regularity assumptions described below.
\begin{assm} \label{asm:assmpt1}
The rewards $r_{t}$ are conditionally independent given $\theta_0$ and $x_0$, or equivalently, given $\theta_0$ and the complete sequence of $\{u_{0}, \ldots, u_{t}, x_0, \ldots, x_t\}$.
\end{assm}
Similar to the independent rewards of the stationary MABs, we have independence of $r_{t} | \{x_t, \theta_0\}$ and $r_{t'} | \{x_{t'}, \theta_0\}$ for $t \neq t'$.
\begin{assm}\label{asm:assmpt3}
The log-likelihood ratio $\ell(r, x, u; \theta, \theta') = \log \frac{p(r | x, u, \theta)}{p(r | x, u, \theta')}$ of $\mathbb{P}_{x, u, \theta}$ is locally $L_{\ell,x}$-Lipschitz continuous with respect to $x$ on the compact set $\mathcal{X}$ for $\theta, \theta' \in \Theta$, $u \in \mathcal{U}$.
\end{assm}
This ensures continuity of the reward distribution with respect to the parameters. If two parameter sets are close to each other in value, then the resulting distributions will also be similar.
\begin{assm}\label{asm:assmpt4}
The distribution $\mathbb{P}_{x, u, \theta}$ for all $x \in \mathcal{X}, u \in \mathcal{U}$, and $\theta \in \Theta$ is sub-Gaussian with parameter $\sigma$, and either $p(r | x, u, \theta)$ has a finite support or $\ell(r, u ;x, \theta, x', \theta')$ is locally $L_{\ell,r}$-Lipschitz with respect to $r$.
\end{assm}
This assumption ensures sample averages are close to their means and is satisfied by many distributions (e.g., Gaussian with known variance). Our last condition ensures the dynamics and the expectation function are well-behaved.
\begin{assm}\label{asm:assmpt5}
Repeated composition of the true dynamics with itself up to $N-1$ times, $f^{t+k}(x_t, u_{t|t}, \ldots, u_{t+k|t}, \theta)$, is Lipschitz continuous with respect to $x_t \in \mathcal{X}$ and $u_{t+k|t} \in \mathcal{U}$ with constants $L_{f,x}$ and $L_{f,u}$, respectively. Besides, the expectation $h(x_t, u_t, \theta)$, for $u_t = \Lambda _t(\mathcal{F}_t)$ in (\ref{eqn:ft}), is Lipschitz continuous with respect to $x_t \in \mathcal{X}$ and $u_t \in \mathcal{U}$ with constants $L_{h,x}$ and $L_{h,u}$, respectively, for all $\theta \in \Theta$. 
\end{assm}
\section{The N-Step Dynamic Regret}\label{sec:Nstepregret}
Our interest is in evaluating the performance of an LBMPC \textit{exploitation} policy for a given $N \leq T$ that is $\Lambda^{E,N}_t(\mathcal{F}_t) = u^*_{t|t}(\widehat{\theta}_t)$ for the corresponding value from the maximizer of $V_N(x_t,\widehat{\theta}_t,t)$ where $\widehat{\theta}_t$ are the control policy's estimates of the unknown $\theta_0$. Data-driven policies are often evaluated by comparing performance to a benchmark policy, and it is typical to benchmark using the optimal policy \cite{garivier2008upperconfidence, besbes2014stochastic, bouneffouf2016multi}. In our setting, the optimal policy is a sequence of functions $\Lambda^*_t(\mathcal{F}_t)_{t=0}^T$ maximizing $\textstyle \sum_{t=0}^T h(x_t,u_t,\theta_0)$ subject to the knowledge available to the control policy (which does not include $\theta_0$). However, computing optimal policies for the problems we consider is PSPACE-hard \cite{papadimitriou1999complexity}. Even their structure is not known for our setup, including for the special case of linear dynamics and quadratic cost function with unknown coefficients.

An alternative benchmark is an oracle policy that has perfect knowledge of $\theta_0$. Specifically, we will use the LBMPC \textit{oracle} policy that is $\Lambda^{O,N}_t(\mathcal{F}_t) = u^*_{t|t}(\theta_0)$ for the corresponding value from the maximizer of $V_N(x_t,\theta_0,t)$ as defined in (\ref{eqn:lbmpc}).  However, there are two subtleties that have to be discussed.

The first subtlety is that the horizon length of the LBMPC oracle policy could potentially be different than the horizon length of the LBMPC policy. However, using different control horizon lengths can lead to different sums of expected rewards over the entire control horizon $\mathcal{T}$. Though this behavior is well known within the MPC community, its implication on evaluating learning-based control policies has not been previously appreciated. The implication is that comparing policies with different horizon lengths leads to a poorly-defined regret notion, and that we should compare oracle policies and learning-based policies with the same finite-horizon. 

The second subtlety is that the presence of nonlinear dynamics in our setup means the state trajectory of a system always controlled by a benchmark policy can be very different than that of a system always controlled by a learning policy, even if the learning policy converges towards the benchmark policy. For this reason, we define a regret notion to compare a finite-horizon benchmark policy to a finite-horizon learning-based policy. We consider an $\epsilon$-greedy policy $\Lambda^{\epsilon,N}_t$ that uses the LBMPC policy $\Lambda^{E,N}_t$ at each greedy exploitation step. Let $x_t,u_t$ be the state and input for the system as controlled by the oracle policy $\Lambda^{O,N}_t$, and let $x_t',u_t'$ be the state and input for the system as controlled by the $\epsilon$-greedy policy $\Lambda^{\epsilon,N}_t$. Then, the expected \emph{$N$-step dynamic regret} is defined as
\begin{equation} \label{dynregret}
    R_{N,T} =  \textstyle \sum_{t=0}^T h(x_t,\Lambda_t^{O,N}(\mathcal{F}_t),\theta_0) -    h(x_t',\Lambda_t^{\epsilon,N}(\mathcal{F}_t'),\theta_0) 
\end{equation}
where $\mathcal{F}_t$ is as defined in (\ref{eqn:ft}) and $\mathcal{F}_t'$ is as defined in (\ref{eqn:ft}) with $x',u'$ replacing $x,u$. This definition is closely related to the traditional dynamic regret \cite{zinkevich2003online, hall2013dynamical}, and the novel aspect of ours is that it compares two $N$-step finite-horizon policies. 
\section{Parameter Estimation}\label{sec:estimate}
Let the variables $\{r_i\}_{i=0}^{t-1}$ be the actual observed values of the rewards up to time $t$. Using Assumption \ref{asm:assmpt1}, the joint likelihood $p(\{r_i\}_{i=0}^{t-1} | x_0, \ldots, x_{t}, u_0, \ldots, u_{t-1}, \theta)$ can be expressed as $\prod_{i=0}^{t-1} p(r_i| x_i, u_i, \theta) P(x_i|x_{i-1},\theta)$. Here, the one step transition likelihood $P(x_i|x_{i-1},\theta)$ is a degenerate distribution with all probability mass at $x_i$, by perpetuation of the dynamics $f(x_i, u_i, \theta)$ with initial conditions $x_{i-1}$. Thus, the maximum likelihood estimator (MLE) for $\theta$ is
\begin{equation}\label{problem}
\begin{aligned}
\widehat{\theta}_t \in \argmin_{\theta \in  \Theta}& \textstyle - \sum_{i=0}^{t-1} \log p(r_i| x_i, u_i, \theta)\\
\mathrm{s.t.}\ & x_{i+1} = f(x_i, u_i,\theta) \ \forall i \in \{0,\ldots,t-1\}
\end{aligned}
\end{equation}
This MLE problem can be computed using optimization, dynamic programming, or various filtering techniques for different problem structures. The Kalman Filter (KF) is a recursive estimator for linear-quadratic discrete-time systems. In more complex systems with non-Gaussian distributions and nonlinear dynamics, the Extended KF and Particle Filter are well-known estimators \cite{kalman1960new, kitagawa1996monte, anderson2012optimal}. For practical purposes, these efficient approaches motivate the use of MLE in our policy. Further, if the controller did not have perfect state measurements, we could use the noisy state data to estimate the dynamics in the constraints of (\ref{problem}) \cite{Amelin2012RandomizedCF, Kalmuk2017OnlinePE}, which would also alleviate any potential infeasibility issues of the MLE.

We further analyze the concentration properties of the solution to (\ref{problem}) and take an approach to the theoretical analysis that generalizes that of \cite{mintz2017non}. We begin by introducing the notion of trajectory Kullback–Leibler (KL) divergence. Since this problem includes the joint distribution of a trajectory of values, the concentration bound for the parameter estimates is computed with regards to the trajectory KL divergence.
\begin{defn} 
    The \textit{trajectory Kullback–Leibler (KL) divergence} between the parameter trajectories $\theta, \theta' \in \Theta$ with the same input sequence $\Pi_T = \{u_t\}_{t=0}^T$ is $\textstyle D_{\Pi_T} (\theta || \theta') = \sum_{i=0}^{T} D_{KL}(\mathbb{P}_{f^i(x_0, \Pi_i,\theta), u_i,\theta} || \mathbb{P}_{f^i(x_0, \Pi_i,\theta'), u_i, \theta'})$, where $\Pi_i$ is the given sequence of control inputs from time $0$ to $i$, $f^i$ is the repeated composition of the dynamics $f$ with itself $i$ times subject to $\Pi_i$, and $D_{KL}$ is the standard KL-Divergence. 
\end{defn}
We have an \textit{observability} assumption with the implication that the distance between two different parameters $\theta, \theta' \in \Theta$ is bounded proportional to their trajectory KL divergence.
\begin{assm}\label{asm:assmpt6} 
For a given input sequence $\Pi_T$ and parameters $\theta \neq \theta'$, if $D_{\Pi_T} (\theta || \theta') \leq \delta$, then $\|\theta - \theta'\| \leq C\delta$ for $C > 0$. 
\end{assm}
We next reformulate the MLE problem (\ref{problem}) by removing the state dynamics constraints through repeated composition of $f$, that is $
    \textstyle \widehat{\theta}_t \in \argmin_{\theta \in \Theta} \textstyle \frac{1}{t-1} \sum_{i=0}^{t-1} \log \frac{p(r_i | f^i(x_0, \Pi_i, \theta_0), u_i,  \theta_0)}{p(r_i | f^i(x_0, \Pi_i, \theta), u_i,  \theta)}$. This reformulation is helpful for our theoretical analysis since for fixed $\theta$, the expected value of the above objective function under $\mathbb{P}_{x_0, \Pi_T, \theta_0}$ is simply $\textstyle \frac{1}{t-1}D_{\Pi_T} (\theta_0 || \theta)$. Hence, we can interpret the MLE problem as minimizing the trajectory KL divergence between the distribution of potential sets of parameters and that of the true parameter set. This interpretation is helpful for us to derive our concentration inequalities. For conciseness of our analysis in this paper, we present the final concentration bound for $\widehat{\theta}_t$ and do not include its proof since it largely follows by the theoretical arguments in \cite{mintz2017non}.
\begin{theorem}\label{mle:thm2}
For any constant $\zeta > 0$, we have the bound that $\textstyle P(\frac{1}{t-1} D_{\Pi_t} (\theta_0 || \widehat{\theta}_t) \leq \zeta + \frac{c_f(d_x, d_\theta)}{\sqrt{t-1}}) \textstyle\geq 1 - \exp (- \frac{\zeta^2 (t-1)}{2L_{\ell,r}^2 \sigma^2} )$ where the constant $c_f(d_x, d_\theta) = 8 L_{f,x} L_{\ell,x} \mathrm{diam}(\mathcal{X}) \sqrt{\pi} + 48\sqrt{2}(2)^{\frac{1}{d_x + d_\theta}} L_{f,x} L_{\ell,x}\mathrm{diam}(\Theta \times \mathcal{X}) \sqrt{\pi (d_x + d_\theta)}$ depends upon $d_x$ and $d_\theta$ (dimensionalities of $\mathcal{X}$ and $\Theta$), and $\mathrm{diam}(\mathcal{X}) = \max_{x, y \in \mathcal{X}} \|x-y\|_2$.
\end{theorem}
\begin{proof} Omitted. Refer to Section 3 of \cite{mintz2017non}.
\end{proof}
We will use this concentration inequality to prove the regret bound of our non-myopic $\epsilon$-greedy policy that we present next.
\section{Proposed Approach} \label{sec:approach}
We develop a \textit{non-myopic $\epsilon$-greedy algorithm} that can achieve effective regret bounds for the non-stationary and nonlinear LBMPC introduced in Section \ref{sec:problemform}. Our choice of algorithm aims to draw a connection between the control and MAB literature. A possible alternative could be adding additive noise to the control inputs which we leave as a future work. When compared with the other well-known MAB strategies, Thompson Sampling (TS) and Upper Confidence Bound (UCB), $\epsilon$-greedy is significantly easier from a computational standpoint for combining with the LBMPC formulation of our non-myopic exploitation problem. TS requires characterization of the posterior distribution which is indeed not possible under the general dynamics considered. Similarly, UCB requires being able to compute the confidence bounds which is not feasible in this framework. Hence, those strategies are not practical for the kinds of applications we are interested in. 

Our Algorithm \ref{alg:nonmyopic} explores randomly according to a non-stationary stochastic process. The initial state $x_0$ is an arbitrary point from the $\mathcal{X}$. At each time $t \in \mathcal{T}$, the algorithm samples a Bernoulli variable $s_t$ based on the exploration probability $\epsilon_t$. If $s_t = 1$, it performs pure exploration. To ensure robust constraint satisfaction and feasibility after exploration, it chooses an input $u_{t|t}$ uniform randomly from $\mathcal{\overline{U}}(x_{t}) = \{u : Ax_{t} + Bu \in \Omega \ominus \mathcal{W}, u \in \mathcal{U}\}$. If $s_t = 0$, the algorithm performs a greedy exploitation step by solving the non-myopic exploitation problem $V_N(x_t, \widehat{\theta}_t, t)$ to select the sequence of inputs with the highest MLE-estimated $N$-step reward. Finally, the algorithm observes the updated state $x_{t+1}$ and reward $r_t$ after applying the chosen input $\Lambda^{\epsilon, N}_t(\mathcal{F}_t)$ to the system. 
\begin{algorithm}
\caption{Non-myopic $\epsilon$-Greedy Algorithm}
\label{alg:nonmyopic} 
\begin{algorithmic}[1]
\State Set: $c > 0$ and $x_0 \in \mathcal{X}$
\For{$t \in \mathcal{T}$} 
\State Set: $\epsilon_t = \min \big\{1, \nicefrac{c}{t}\big\}$ 
\State Sample: $s_t \sim \text{Bernoulli} (\epsilon_t)$
\If {$s_t = 1$}
\State Randomly select: $u_{t|t} \in \mathcal{\overline{U}}(x_{t})$
\State Set: $\Lambda^{\epsilon, N}_t(\mathcal{F}_t) = u_{t|t}$
\Else {} 
\State Compute: $\widehat{\theta}_t$ from (\ref{problem}) 
\State Compute: $u^*_{t|t}(\widehat{\theta}_t)$ from $V_N(x_t, \widehat{\theta}_t, t)$ (\ref{eqn:lbmpc})
\State Set: $\Lambda^{\epsilon, N}_t(\mathcal{F}_t) = u^*_{t|t}(\widehat{\theta}_t)$
\EndIf
\State Observe: $r_t$ and $x_{t+1}$
\EndFor
\end{algorithmic}
\end{algorithm}
\begin{remark}
If $\mathcal{W}, \mathcal{X}, \mathcal{U}$ are all polytopes, then $\Omega$ can be approximated by a polytope arbitrarily well. Then, $\Omega \ominus \mathcal{W}$ is also a polytope. As a result, line 6 involves randomly picking an element from a polytope that can be done in a computationally efficient way using standard algorithms. 
\end{remark}
For clarity, we consider a randomization at the initial system state, and then assume noise-free transitions for the subsequent states which is common in the line of RL for finite finite sample analysis \cite{bertsekas2010distributed, lazaric2010analysis, liu2014, fazel2018global}. Our analysis here provides a strong ground for generalization of our policy to the setting of imperfect state measurements as an important direction for future work. Note that the exploration probability $\epsilon_t$ decays over time. This reduces the cost of exploration by ensuring the algorithm makes fewer unnecessary explorations as more data collected and the estimates of our policy improve.
\subsection{Lipschitzian Stability of Non-myopic Exploitation} \label{subsec:stability}
We prove Lipschitzian stability, with respect to perturbations of parameter values, of optimal solutions of the non-myopic exploitation policy $\Lambda^{E, N}_t(\mathcal{F}_t) = u^*_{t|t}(\widehat{\theta}_t)$ by proving a second order growth condition and Lipschitz continuity of the difference of the perturbed and unperturbed objective functions.

\begin{lem}\label{lem4}
Suppose $U_{N,t} = \{u_{t|t}, \ldots, u_{t+N|t}\}$ is a feasible input sequence for $V_N(x_t, \widehat{\theta}_t, t)$. Let $J_N(x_t, U_{N,t}, \widehat{\theta}_t, t)$ be the estimated $N$-step reward of this input sequence at time $t$, i.e., 
\begin{equation}
    \textstyle J_N(x_t, U_{N,t}, \widehat{\theta}_t, t) = \sum_{k=0}^{N} h(\tilde{x}_{t+k|t}, u_{t+k|t}, \widehat{\theta}_t)  \label{defn:J}
\end{equation}
where $\tilde{x}_{t+k+1|t} = f(\tilde{x}_{t+k|t}, u_{t+k|t}, \widehat{\theta}_t)$ for $k \in \langle N-1 \rangle$ as given in (\ref{eqn:lbmpc}). Then, $J_N(x_t, U_{N,t}, \widehat{\theta}_t, t)$ is $(L_{f,u} \cdot L_{h,u})$-Lipschitz continuous with respect to $U_{N,t}$ on the compact set $\mathcal{U}^{N+1}$ for any feasible input sequence $U'_{N,t} = \{u'_{t|t}, \ldots, u'_{t+N|t}\}$. 
\end{lem}
\begin{proof}
By Assumption \ref{asm:assmpt5}, $f^{t+k}(x_t, u_{t|t}, \ldots, u_{t+k|t}, \widehat{\theta})$ is $L_{f,u}$-Lipschitz continuous and $h(\tilde{x}_{t+k|t}, u_{t+k|t}, \widehat{\theta}_t)$ is $L_{h,u}$-Lipschitz continuous with respect to $u_{t+k|t} \in \mathcal{U}$. Then, by preservation of Lipschitz continuity across functional compositions and addition, we have the desired condition.
\end{proof}

Lemma \ref{lem4} implies the second order growth condition for $V_N(x_t, \widehat{\theta}_t, t)$ since it shows $J_N$ increases at least linearly over a compact set. We next present the second condition required for the Lipschitzian stability of the maximizer of $V_N(x_t, \widehat{\theta}_t, t)$.

\begin{assm}\label{differencefunc} Let $U^*_{N,t}(\widehat{\theta}_t) = \{u^*_{t|t}(\widehat{\theta}_t), \ldots, u^*_{t+N|t}(\widehat{\theta}_t)\}$ and $U^*_{N,t}(\theta) = \{u^*_{t|t}(\theta), \ldots, u^*_{t+N|t}(\theta)\}$ be maximizers of $V_N(x_t, \widehat{\theta}_t, t)$ and $V_N(x_t, \theta, t)$. Then for $\kappa \geq 0$, 
    $| [J_N(x_t, U^*_{N,t}(\widehat{\theta}_t), \widehat{\theta}_t, t) - J_N(x_t, U^*_{N,t}(\widehat{\theta}_t), \theta, t)]   - [J_N(x_t, U^*_{N,t}(\theta), \widehat{\theta}_t, t) - J_N(x_t, U^*_{N,t}(\theta), \theta, t)]|  \leq \kappa \|\widehat{\theta}_t - \theta\| \cdot \|U^*_{N,t}(\widehat{\theta}_t)-U^*_{N,t}(\theta)\|$.
\end{assm} 
We now give a sufficient condition for Assumption \ref{differencefunc}.
\begin{prop}\label{suffcond} For any $\theta \in \Theta$ and real constant $L_J \geq 0$, if $\|\nabla_u J_N(x_t, U^*_{N,t}(\widehat{\theta}_t), \widehat{\theta}_t, t) - \nabla_u J_N(x_t, U^*_{N,t}(\widehat{\theta}_t), \theta, t)\|_{\infty} 
    \leq L_J \|\widehat{\theta}_t - \theta\|$ holds, then Assumption \ref{differencefunc} is satisfied.
\end{prop}
\begin{proof}
Let $s(\tau) = U^*_{N,t}(\widehat{\theta}_t) + \tau\cdot(U^*_{N,t}(\theta) - U^*_{N,t}(\widehat{\theta}_t))$. This implies $ s(0) = U^*_{N,t}(\widehat{\theta}_t)$ and $s(1) = U^*_{N,t}(\theta)$. Then, 
\begin{align}
&[J_N(x_t, U^*_{N,t}(\widehat{\theta}_t), \widehat{\theta}_t, t) - J_N(x_t, U^*_{N,t}(\widehat{\theta}_t), \theta, t)] -  [J_N(x_t, U^*_{N,t}(\theta), \widehat{\theta}_t, t) - J_N(x_t, U^*_{N,t}(\theta), \theta, t)] \nonumber \allowdisplaybreaks \\
& = \textstyle \int_0^1 \nabla_U J(x_t, s(\tau), \widehat{\theta}_t,t)^T (U^*_{N,t}(\theta)-U^*_{N,t}(\widehat{\theta}_t))d\tau - \textstyle \int_0^1 \nabla_U J(x_t, s(\tau), \theta,t)^T (U^*_{N,t}(\theta)-U^*_{N,t}(\widehat{\theta}_t))d\tau \\
\intertext{where the last equality follows by the Fundamental Theorem of Calculus for Line Integrals. Then, we continue as}
&= | \textstyle  \int_0^1 [\nabla_U J(x_t, s(\tau), \widehat{\theta}_t,t) - \nabla_U J(x_t, s(\tau), \theta,t)]^T  (U^*_{N,t}(\theta)-U^*_{N,t}(\widehat{\theta}_t))d\tau| \allowdisplaybreaks \\
&\leq  \textstyle \int_0^1 \| \nabla_U J(x_t, s(\tau), \widehat{\theta}_t,t) - \nabla_U J(x_t, s(\tau), \theta,t) \|_\infty \big|\big|U^*_{N,t}(\theta)-U^*_{N,t}(\widehat{\theta}_t)\|_1 d\tau \label{holder} \allowdisplaybreaks \\
&\leq L_J \|\widehat{\theta}_t - \theta\| \cdot \|U^*_{N,t}(\theta)-U^*_{N,t}(\widehat{\theta}_t)\|_1 \allowdisplaybreaks \label{property} \\
&\leq \sqrt{N} L_J \|\widehat{\theta}_t - \theta\| \cdot \|U^*_{N,t}(\theta)-U^*_{N,t}(\widehat{\theta}_t)\|_2 \allowdisplaybreaks
\end{align}
where (\ref{holder}) follows by H\"{o}lder's inequality, and (\ref{property}) follows by the assumed property in Proposition \ref{suffcond}. This gives us the desired result in Assumption \ref{differencefunc} by setting $\kappa = \sqrt{N}L_J$.
\end{proof}
\begin{lem}\label{lem:suffcond}
If the state dynamics $f(x, u, \theta)$ and the expectation function $h(x, u,\theta)$ are polynomial functions, then the sufficient condition given in Proposition \ref{suffcond} holds.
\end{lem}
\begin{proof}
Since (\ref{defn:J}) is the average of compositions of two polynomials $f$ and $h$, it is polynomial. Then, $\nabla_u J_N(x,U,\theta,t)$ is polynomial on the bounded domain $\mathcal{X} \times \mathcal{U}^{N+1} \times \Theta$. Hence, by Corollary 8.2 in \cite{estep2010practical}, $\nabla_u J_N(x,U,\theta,t)$ is locally Lipschitz with respect to $\theta \in \Theta$ for any $x \in \mathcal{X}, U \in \mathcal{U}^{N+1}, t \in \mathcal{T}$.
\end{proof}
A specific example where Lemma \ref{lem:suffcond} holds is a discrete-time linear time-invariant system with $f(x,u,\theta) = Ax + Bu$ and $h(x,u,\theta) = x^T Q x + u^T R u$ where $\theta = [Q, R, A, B]$. 
\begin{lem} \label{lem:prop4.32}
If Assumption \ref{differencefunc} and Lemma \ref{lem4} hold, then the Lipschitzian stability property follows by Proposition 4.32 in \cite{bonnans2013perturbation}, i.e., $\big|\big|U^*_{N,t}(\widehat{\theta}_t) - U^*_{N,t}(\theta) \big|\big| \leq c_u^{-1} \kappa \|\widehat{\theta}_t - \theta\|$ for $c_u > 0$.
\end{lem}
Since $\big|\big|u^*_{t|t}(\widehat{\theta}_t) - u^*_{t|t}(\theta) \big|\big| \leq \big|\big|U^*_{N,t}(\widehat{\theta}_t) - U^*_{N,t}(\theta) \big|\big|$, we conclude that the non-myopic exploitation policy $\Lambda^{E, N}_t(\mathcal{F}_t) = u^*_{t|t}(\widehat{\theta}_t)$ corresponding from the maximizer of $V_N(x_t, \widehat{\theta}_t, t)$ is $c_u^{-1} \kappa$-Lipschitz continuous with respect to $\hat{\theta}_t \in \Theta$.
\subsection{Regret Analysis}\label{subsec:regret}
We next characterize the $N$-step dynamic regret $R_{N,T}$ (\ref{dynregret}) of Algorithm \ref{alg:nonmyopic}. By definition, $R_{N,T}$ compares the LBMPC oracle policy $\Lambda_t^{O,N}(\mathcal{F}_t)$ for the system $x_t, u_t$ as controlled by the oracle policy to our non-myopic $\epsilon$-greedy policy $\Lambda_t^{\epsilon,N}(\mathcal{F}'_t)$ for the system $x_t', u'_t$ as controlled by the learning-policy that uses the LBMPC policy $\Lambda_t^{E,N}(\mathcal{F}'_t)$ at greedy exploitation steps. We start by bounding a weaker notion that compares the actions chosen under the states $x_t'$ achieved by $\Lambda_t^{\epsilon,N}(\mathcal{F}'_t)$. 
\begin{theorem} \label{thm:nstep} The non-myopic $\epsilon$-greedy policy $\Lambda_t^{\epsilon,N}(\mathcal{F}'_t)$ and the LBMPC oracle policy $\Lambda_t^{O,N}(\mathcal{F}'_t)$ satisfy the following result for the system states $x_t'$ that are achieved by $\Lambda_t^{\epsilon,N}(\mathcal{F}'_t)$: 
\begin{align}
    &\textstyle\sum_{t=0}^T h(x_t',\Lambda^{O,N}(\mathcal{F}_t'),\theta_0) - \textstyle\sum_{t=0}^T h(x_t',\Lambda^{\epsilon,N}(\mathcal{F}_t'),\theta_0) \nonumber \allowdisplaybreaks  \\
    &\leq \mathcal{M} \exp \Big(\textstyle \frac{c^2_f(d_x, d_\theta) }{2L_{\ell,r}^2 \sigma^2}\Big) (\mathcal{C} + \log T ) + \mathcal{M}c (1 - \log (c+1) + \log T) + \textstyle \frac{L_{h,u} \kappa C \sqrt{4 L_{\ell,r}^2 \sigma^2}}{c_u} \sqrt{T} \log T \allowdisplaybreaks \label{eqn:nstepregretbound}
\end{align}
where $C>0$, $c_f (d_x, d_\theta)$ is the constant in Theorem \ref{mle:thm2}, and $\mathcal{C}$ is a bound on the finite summation $\textstyle \sum_{t=1}^9 \exp (- (\log t)^2)$.
\end{theorem}
\begin{proof} For notational convenience, let $\mathbb{E}[M_t] = h(x_t',\Lambda^{O,N}(\mathcal{F}_t'),\theta_0) - h(x_t',\Lambda^{\epsilon,N}(\mathcal{F}_t'),\theta_0)$. Let $\mathcal{T}^{\text{xit}} \in \mathcal{T}$ and $\mathcal{T}^{\text{xre}} \in \mathcal{T}$ be the set of random time points that Algorithm (\ref{alg:nonmyopic}) performs exploitation and exploration, respectively. Noticing the cardinalities $\# \mathcal{T}^{\text{xit}}$, $\# \mathcal{T}^{\text{xre}}$ are random variables, we have $
    \textstyle \sum_{t=0}^T \mathbb{E}[M_t] \textstyle = \textstyle\sum_{t \in \mathcal{T}^{\text{xit}}} h(x_t',u^*_{t|t}(\theta_0),\theta_0) - h(x_t',u^*_{t|t}(\widehat{\theta}_t),\theta_0)   + \textstyle \sum_{t \in \mathcal{T}^{\text{xre}}} h(x_t',u^*_{t|t}(\theta_0),\theta_0) - h(x_t',u_{t|t},\theta_0)$. We note that $\mathbb{E}[M_t]$ is a bounded value since $\mathcal{X}, \Theta, \mathcal{U}$ are all compact sets and $h(x, u, \theta)$ is a bounded continuous function on this domain. Then, assuming $\mathbb{E}[M_t] \leq \mathcal{M}$, we obtain 
\begin{align}
    &\textstyle [\sum_{t=0}^T \mathbb{E}[M_t]  | \mathcal{T}^{\text{xit}} ] \leq \mathcal{M}\mathbb{E} [\# \mathcal{T}^{\text{xre}} ] + \textstyle  \sum_{t \in \mathcal{T}^{\text{xit}}}  h(x_t',u^*_{t|t}(\theta_0),\theta_0) - h(x_t',u^*_{t|t}(\widehat{\theta}_t),\theta_0)\allowdisplaybreaks \label{bound0} 
\end{align}
We can rewrite each term inside the summation above as
\begin{align}    
    h(x_t',u^*_{t|t}(\theta_0),\theta_0) - h(x_t',u^*_{t|t}(\widehat{\theta}_t),\theta_0) &= \mathbb{E}[M_t |D_{\Pi_t}(\theta_0||\widehat{\theta}_t) \leq \delta_{\widehat{\theta}_t}, x_t, \theta_0, \widehat{\theta}_t] P(D_{\Pi_t}(\theta_0||\widehat{\theta}_t) \leq \delta_{\widehat{\theta}_t}) \notag \allowdisplaybreaks \\
    &\quad + \mathbb{E}[M_t |D_{\Pi_t}(\theta_0||\widehat{\theta}_t) \geq \delta_{\widehat{\theta}_t}, x_t, \theta_0, \widehat{\theta}_t] P(D_{\Pi_t}(\theta_0||\widehat{\theta}_t) \geq \delta_{\widehat{\theta}_t})  \allowdisplaybreaks \\
    &= (13,a) + (13,b) \allowdisplaybreaks \label{bound1}
\end{align}
Let $\varepsilon(\delta_{\widehat{\theta}_t}) = \max \{ \|\theta_0 - \widehat{\theta}_t\| :D_{\Pi_t}(\theta_0||\widehat{\theta}_t) \leq \delta_{\widehat{\theta}_t} \}, \forall t \in \mathcal{T}^{\text{xit}}$. 
\begin{align}
    \textstyle\sum_{t \in \mathcal{T}^{\text{xit}}}  (13,a) &= \textstyle\sum_{t \in \mathcal{T}^{\text{xit}}} h(x_t',u^*_{t|t}(\theta_0),\theta_0) - h(x_t',u^*_{t|t}(\widehat{\theta}_t),\theta_0) \allowdisplaybreaks \\
    &\leq \textstyle\sum_{t \in \mathcal{T}^{\text{xit}}} L_{h,u} \|u^*_{t|t}(\theta_0) - u^*_{t|t}(\widehat{\theta}_t)\| \allowdisplaybreaks\label{byasm5} \\
    &\leq \textstyle\sum_{t \in \mathcal{T}^{\text{xit}}} L_{h,u} \|U^*_{N|t}(\theta_0) - U^*_{N|t}(\widehat{\theta}_t)\| \allowdisplaybreaks \\
    &\leq \textstyle \frac{L_{h,u} \kappa}{c_u} \textstyle\sum_{t \in \mathcal{T}^{\text{xit}}} \|\theta_0 - \widehat{\theta}_t\| \allowdisplaybreaks \label{book}\\
    &\leq \textstyle \frac{L_{h,u} \kappa}{c_u} \textstyle\sum_{t \in \mathcal{T}^{\text{xit}}} \varepsilon(\delta_{\widehat{\theta}_t}) \intertext{where (\ref{byasm5}) follows by Assumption \ref{asm:assmpt5} and (\ref{book}) follows by Lemma \ref{lem:prop4.32}. 
     Now, we have $\varepsilon(\delta_{\widehat{\theta}_t}) = C \delta_{\widehat{\theta}_t}$ for a constant $C>0$ by Assumption \ref{asm:assmpt6} and let $\eta(t) = |\{s \in \mathcal{T}^{\text{xit}}: s \leq t \}|$. Then, for $\delta_{\widehat{\theta}_t} = O\Big(\textstyle \sqrt{4 L_{\ell,r}^2 \sigma^2} \log \eta(t) / \sqrt{\eta(t)}\Big)$, we obtain}
    &\leq \textstyle \frac{L_{h,u} \kappa C \sqrt{4 L_{\ell,r}^2 \sigma^2}}{c_u} \sqrt{\# \mathcal{T}^{\text{xit}}} \log \# \mathcal{T}^{\text{xit}} \label{exploitbound}\\
    &\leq \textstyle \frac{L_{h,u} \kappa C \sqrt{4 L_{\ell,r}^2 \sigma^2}}{c_u} \sqrt{T} \log T  \allowdisplaybreaks
\end{align}
To bound the second term in (\ref{bound1}), recall $\mathbb{E}[M_t] \leq \mathcal{M}$. Then, 
\begin{align}
    \textstyle\sum_{t \in \mathcal{T}^{\text{xit}}}  (13,b) &\leq \mathcal{M} \textstyle\sum_{t \in \mathcal{T}^{\text{xit}}} \exp \Big(\frac{- (\delta_{\widehat{\theta}_t} \sqrt{t-1} - c_f(d_x, d_\theta) )^2}{2L_{\ell,r}^2 \sigma^2} \Big) \label{concbound}  \allowdisplaybreaks \\
    &\leq \mathcal{M} \textstyle\sum_{t \in \mathcal{T}^{\text{xit}}} \exp \Big(\frac{- \delta^2_{\widehat{\theta}_t}(t-1)/2 + c^2_f(d_x, d_\theta) }{2L_{\ell,r}^2 \sigma^2} \Big)  \allowdisplaybreaks \\
    &\leq \mathcal{M} \exp \Big(\textstyle \frac{c^2_f(d_x, d_\theta) }{2L_{\ell,r}^2 \sigma^2} \Big) \Big(\textstyle\sum_{t=1}^9 \exp(-(\log t)^2) +\textstyle \sum_{t \in \mathcal{T}^{\text{xit}}, t \geq 10} \exp (- \log t)\Big) \allowdisplaybreaks \\
    &\leq \mathcal{M} \exp \Big(\textstyle \frac{c^2_f(d_x, d_\theta)}{2L_{\ell,r}^2 \sigma^2} \Big) (\mathcal{C} + \log T )
\end{align}
where (\ref{concbound}) follows by Theorem \ref{mle:thm2} and $\mathcal{C}$ can be approximated as 2.2232. Lastly, we bound the first term in (\ref{bound0}): $\mathcal{M}\mathbb{E} [\# \mathcal{T}^{\text{xre}}] =\mathcal{M} \textstyle \sum_{t =0}^T \min \{1, \frac{c}{t}\} \leq \mathcal{M}  \Big(c + \textstyle \sum_{t=c+1}^T \frac{c}{t} \Big) \allowdisplaybreaks \leq \mathcal{M}c (1 - \log (c+1) + \log T)$. Substituting these into (\ref{bound0}):
\begin{align}
    &\textstyle [\sum_{t=0}^T \mathbb{E}[M_t]  | \mathcal{T}^{\text{xit}} ] \leq \mathcal{M} \exp \Big(\frac{c^2_f(d_x, d_\theta) }{2L_{\ell,r}^2 \sigma^2} \Big) (\mathcal{C} + \log T ) + \mathcal{M}c (1 - \log (c+1) + \log T) + \textstyle \frac{L_{h,u} \kappa C \sqrt{4 L_{\ell,r}^2 \sigma^2}}{c_u} \sqrt{T} \log T 
 \end{align}
and taking the expectation gives us the desired result. 
\end{proof}
We analyze regret of $\Lambda^{\epsilon,N}(\mathcal{F}_t')$ by assuming stability of $\Lambda^{O,N}_t(\mathcal{F}_t)$. If the LBMPC from Sect. \ref{sec:problemform} does not provide stability, the full LBMPC formulation \cite{aswani2013provably} can achieve stability. Our results in this paper generalize to the full formulation but at the expense of substantial notational complexity. 
\begin{assm} \label{expstable}
Let $x_{eq} \in \Omega$ be an equilibrium for the LBMPC system in Sect. \ref{sec:problemform}. For $\alpha \in [0, 2/3]$ and $\mathcal{F}_t$ as in (\ref{eqn:ft}), the LBMPC oracle policy $\Lambda^{O,N}_t(\mathcal{F}_t)$ satisfies $\| Ax_t + B\Lambda^{O,N}_t(\mathcal{F}_t) + g(x_t, \Lambda^{O,N}_t(\mathcal{F}_t), \theta_0) - x_{eq} \|
    \leq \alpha \| x_t - x_{eq} \| \ \forall t$. 
\end{assm}
Exponential stability of the nonlinear LBMPC implied by this assumption can be ensured under certain sufficient conditions established in the literature \cite{mayne2000constrained, pannocchia2011conditions}. Generalizing the results with less restrictive stability notions poses future research.
\begin{theorem} 
For $4 \leq c \leq \sqrt[4]{T}/3$, the expected $N$-step dynamic regret $R_{N,T}$ (\ref{dynregret}) for a policy $\Lambda^{\epsilon,N}(\mathcal{F}_t')$ computed by Algorithm \ref{alg:nonmyopic} satisfies 
\begin{multline}
    R_{N,T}\leq \textstyle 2L_{h,x}\sqrt{T} \mathrm{diam}(\mathcal{X}) + \frac{2L_{h,x} c (3 - \alpha)}{1 - \alpha} \mathrm{diam}(\mathcal{X})\log T + \frac{4L_{h,x}\overline{C} c^2}{\alpha} \sqrt{T} (\log T)^3  \\ 
    + \mathcal{M} \exp \Big(\textstyle \frac{c^2_f(d_x, d_\theta) }{2L_{\ell,r}^2 \sigma^2}\Big) (\mathcal{C} + \log T ) + \mathcal{M}c (1 - \log (c+1) + \log T) + \textstyle \frac{L_{h,u} \kappa C \sqrt{4 L_{\ell,r}^2 \sigma^2}}{c_u} \sqrt{T} \log T
\end{multline} with probability at least 
\begin{equation}
    \textstyle 1 - (T-2\sqrt{T})\exp \Big(- \frac{4c^2 \left(\log \frac{e (2\sqrt{T}+2}{c+1}\right)^2}{2 c\log (2\sqrt{T}+1) + \frac{2c^2}{2\sqrt{T}+1} + \frac{4c^2}{3}\log \frac{e (2\sqrt{T}+2)}{c+1}}\Big)  - \exp \Big(- \frac{c^2 \big(\log \frac{T}{2\sqrt{T}+1}\big)^2}{(4+ \frac{2}{3}c^2) \log T}\Big) 
\end{equation}
where $\overline{C} = c_u^{-1}(\|B\| + L_{f,u}) \kappa C \sqrt{4 L^2_{\ell, r} \sigma^2}$. 
\end{theorem}
\begin{proof}
By Assumption \ref{asm:assmpt5} and the upper bound in (\ref{eqn:nstepregretbound}),  
\begin{align}
R_{N,T} &= \textstyle\sum_{t=0}^T h(x_t,\Lambda_t^{O,N}(\mathcal{F}_t),\theta_0) -  h(x'_t,\Lambda_t^{O,N}(\mathcal{F}'_t),\theta_0) + \textstyle\sum_{t=0}^T h(x'_t,\Lambda_t^{O,N}(\mathcal{F}'_t),\theta_0) -  h(x_t',\Lambda_t^{\epsilon,N}(\mathcal{F}_t'),\theta_0) \allowdisplaybreaks  \\
    &\leq L_{h,x} \textstyle\sum_{t=0}^T\|x_t - x_t'\| \ + \  (\ref{eqn:nstepregretbound}) \label{eqn:dyn0} 
\end{align}
Algorithm (\ref{alg:nonmyopic}) performs exploration at random times according to a non-stationary stochastic process over $\mathcal{T}$. We divide $\mathcal{T}$ into \enquote{inter-explore intervals} composed of an exploration and the subsequent exploitations until the next one is reached. Let $I_k = [\underline{I}_k, \overline{I}_k]$ be the $k^{\text{th}}$ sub-interval such that $I_{-1} = [0, 2\lceil \sqrt{T} \rceil]$, $I_0 = [ 2\lceil \sqrt{T} \rceil + 1, t_1^{\text{xre}} - 1 ]$, $I_{k} = [t_k^{\text{xre}}, t_{k+1}^{\text{xre}} - 1]$ for $k \in [1, K-1]$ where $t_k^{\text{xre}}$ is the $k^{\text{th}}$ exploration step after time $2\lceil \sqrt{T} \rceil$, and $I_K = [t_K^{\text{xre}},T]$ where $K = \textstyle \sum_{t = 2\lceil \sqrt{T} \rceil + 1}^T s_t$ and $s_t \sim \mathrm{Bernoulli} \left(\min \big\{1, \nicefrac{c}{t}\big\}\right)$. Then, $\textstyle\sum_{t=0}^{T} \|x_t - x_t'\| = \textstyle \sum_{k = -1}^{K} \sum_{t \in I_k} \|x_t - x_t'\|$. The key idea is that regret over each $I_k, k \in [0, K]$ is bounded above by the regret over $S_k = [\underline{S}_k, \overline{S}_k] = [\underline{I}_k, T]$ that includes a single exploration at time $\underline{I}_k$ followed by exploitation steps thereafter up to $T$. 

Suppose Algorithm \ref{alg:nonmyopic} uses  $\Lambda^{O,N}_t(\mathcal{F}_t)$ at all greedy exploitation steps of $S_k, k \in [0, K]$. Since $x_t \in \mathcal{X}$ for $t \in \mathcal{T}^{\text{xre}}$ and $\mathcal{X}$ is compact, $\|x_t-x_{eq}\| \leq  \mathrm{diam}(\mathcal{X}), \ t \in \mathcal{T}^{\text{xre}}$. Then, by Assumption \ref{expstable},  $\textstyle\sum_{t \in I_k }\|x_t - x_{eq}\| \leq \textstyle\sum_{t \in S_k }\|x_{t} - x_{eq}\| = \textstyle\|x_{t_k^{\text{xre}}} - x_{eq}\| + \textstyle \sum_{t=\underline{S}_k+1}^{\overline{S}_k} \|x_t - x_{eq}\| \leq \textstyle \mathrm{diam}(\mathcal{X}) + \textstyle \sum_{t=\underline{S}_k+1}^{\overline{S}_k} \alpha^{t-\underline{S_k}}  \mathrm{diam}(\mathcal{X}) \leq \textstyle  \mathrm{diam}(\mathcal{X})/ (1 - \alpha)$. Next, suppose instead $\Lambda^{E,N}_t(\mathcal{F}'_t)$ is used at all greedy exploitation steps of $S_k, k \in [0, K]$. Observe the convergence of $\Lambda^{E,N}(\mathcal{F}_t')$: 
\begin{align}
     &\|Ax'_t +  B \Lambda^{E,N}_t(\mathcal{F}'_t) + g(x'_t, \Lambda^{E,N}_t(\mathcal{F}'_t), \theta_0) - x_{eq} \| \nonumber  \allowdisplaybreaks \\
     &\leq \|Ax'_t + B \Lambda^{O,N}_t(\mathcal{F}'_t) + g(x'_t, \Lambda^{O,N}_t(\mathcal{F}'_t), \theta^o) - x_{eq} \| \notag \allowdisplaybreaks\\
     &\quad + \|B \Lambda^{E,N}_t(\mathcal{F}'_t) + g(x'_t, \Lambda^{E,N}_t(\mathcal{F}'_t), \theta_0) - B \Lambda^{O,N}_t(\mathcal{F}'_t) - g(x'_t, \Lambda^{O,N}_t(\mathcal{F}'_t), \theta_0)\| \allowdisplaybreaks \label{eqn:dyn3} \\
     &\leq \alpha \|x'_t - x_{eq} \| + \|B \Lambda^{E,N}_t(\mathcal{F}'_t) + g(x'_t, \Lambda^{E,N}_t(\mathcal{F}'_t), \theta_0)- B \Lambda^{O,N}_t(\mathcal{F}'_t) - g(x'_t, \Lambda^{O,N}_t(\mathcal{F}'_t), \theta_0)\| \allowdisplaybreaks \label{eqn:dyn4} 
\end{align}
where (\ref{eqn:dyn3}) follows by the triangle inequality and (\ref{eqn:dyn4}) follows by Assumption \ref{expstable}. Recall that $\textstyle    \|\Lambda^{E,N}_t(\mathcal{F}'_t) - \Lambda^{O,N}_t(\mathcal{F}'_t)\| \leq \frac{ \kappa C \sqrt{4 L^2_{\ell, r} \sigma^2}}{c_u}\frac{\log \eta(t)}{\sqrt{\eta(t)}}$ as followed from (\ref{byasm5}) to (\ref{exploitbound}). By Assumption \ref{asm:assmpt5}, we have $(\ref{eqn:dyn4}) \textstyle\leq \textstyle\alpha\|x'_t - x_{eq}\| + \overline{C}\frac{\log \eta(t)}{\sqrt{\eta(t)}}$, where $\overline{C} = \frac{(\|B\| + L_{f,u}) \kappa C \sqrt{4 L^2_{\ell, r} \sigma^2}}{c_u}$. Then, for $k \in [0, K]$, 
\begin{align}
 \textstyle \sum_{t \in I_k} \|x'_t - x_{eq}\| \leq \textstyle\sum_{t \in S_k }\|x'_t - x_{eq}\| &= \textstyle\|x'_{t_k^{\text{xre}}} - x_{eq}\| + \textstyle \sum_{t=\underline{S}_k+1}^{\overline{S}_k} \|x'_t - x_{eq}\| \allowdisplaybreaks \\
 &\leq\textstyle \mathrm{diam}(\mathcal{X}) + \textstyle \sum_{t=t_k^{\text{xre}}+1}^{T} \Big[ \alpha^{t-1}  \mathrm{diam}(\mathcal{X}) + \overline{C} \sum_{i=t_k^{\text{xre}}+1}^{t-1} \frac{\alpha^{t-1-i}\log \eta(i)}{\sqrt{\eta(i)}} \Big]  \allowdisplaybreaks \\
 &\leq  \textstyle \frac{2 - \alpha}{1 - \alpha}\mathrm{diam}(\mathcal{X}) + \textstyle \overline{C} \sum_{t=t_k^{\text{xre}}+1}^{T} \sum_{i=t_k^{\text{xre}}+1}^{t-1} \alpha^{t-1-i} \frac{\log \eta(i)}{\sqrt{\eta(i)}}  \allowdisplaybreaks 
 \intertext{Recall $\eta(i) = i - \sum_{j = 1}^{i} s_j$ where $s_j \sim \mathrm{Bernoulli}(\min \{1, \nicefrac{c}{j}\} )$ and $\mathbb{E} \sum_{j = 1}^{i} s_j \leq c + \int_{j = c}^{i} \frac{c}{j}dj = c\log \frac{e i}{c}$. By conditioning on the event 
 $\mathcal{E}_i = \{\sum_{j = 1}^{i} s_j \leq  3\mathbb{E} \sum_{j = 1}^{i} s_j\}$, we get $\eta(i) \geq i -  3\mathbb{E} \sum_{j = 1}^{i} s_j \geq i - 3c \log \frac{e i}{c} \geq \frac{i}{c^2}$ where the last inequality holds for all $i \geq 2\lceil \sqrt{T} \rceil + 1 \geq 6c^2 + 1$. Then, for $\alpha \in [0, 2/3]$, }
  &\leq \textstyle \frac{2 - \alpha}{1 - \alpha} \mathrm{diam}(\mathcal{X}) +  \overline{C} c \sum_{t=t_k^{\text{xre}}+1}^{T} \sum_{i=t_k^{\text{xre}}+1}^{t-1} \alpha^{t-i-1} \frac{\log i}{\sqrt{i}}  
 \allowdisplaybreaks \\
&\leq \textstyle \frac{2 - \alpha}{1 - \alpha} \mathrm{diam}(\mathcal{X}) +  \frac{\overline{C} c}{\alpha} \sum_{t=t_k^{\text{xre}}+1}^{T} \sum_{i=t_k^{\text{xre}}+1}^{t-1} \frac{\log i}{(t-i)\sqrt{i}}   \allowdisplaybreaks \\
&\leq \textstyle \frac{2 - \alpha}{1 - \alpha} \mathrm{diam}(\mathcal{X}) +  \frac{2\overline{C} c}{\alpha} \sum_{t=t_k^{\text{xre}}+1}^{T} \frac{(\log t)^2}{\sqrt{t}}   \allowdisplaybreaks \\
 &\leq \textstyle \frac{2 - \alpha}{1 - \alpha} \mathrm{diam}(\mathcal{X}) +  \frac{2\overline{C} c}{\alpha} \sqrt{T} (\log T)^2 \label{uppepsilon} 
\end{align}
Note $\mathbb{E} \sum_{j = 1}^{i} s_j \geq c \log \frac{e (i+1)}{c+1}$ and $\mathrm{Var}(\sum_{j = 1}^{i} s_j) = \sum_{j = 1}^{i} \frac{c}{j} \cdot \frac{j-c}{j} \leq c\log i + \frac{c^2}{i}$. Then, by Bernstein's inequality \cite{boucheron2013concentration}, (\ref{uppepsilon}) holds with 
\begin{align}
    \mathbb{P}(\cap_{i = t_k^{\text{xre}}+1}^{T} \mathcal{E}_i) \geq \textstyle 1 - \sum_{i = t_k^{\text{xre}}+1}^{T}  P(\overline{\mathcal{E}}_i) &\geq \textstyle 1 - \sum_{i = t_k^{\text{xre}}+1}^{T} \exp \Big(- \frac{4c^2 (\log \frac{e (i+1)}{c+1})^2}{2 c\log i + \frac{2c^2}{i} + \frac{4c^2}{3}\log \frac{e (i+1)}{c+1}}\Big) \allowdisplaybreaks\\
    &\geq \textstyle 1 - (T-2\sqrt{T})\exp \Big(- \frac{4c^2 (\log \frac{e (2\sqrt{T}+2}{c+1})^2}{2 c\log (2\sqrt{T}+1) + \frac{2c^2}{2\sqrt{T}+1} + \frac{4c^2}{3}\log \frac{e (2\sqrt{T}+2)}{c+1}}\Big).
\end{align} The above bounds for $\Lambda^{O,N}(\mathcal{F}_t)$ and (\ref{uppepsilon}) for $\Lambda^{\epsilon,N}(\mathcal{F}_t')$ allow us to bound the deviation of the system trajectory under the learning policy from the one under the oracle policy over $I_{-1}$ as $\textstyle\sum_{t \in I_{-1}} \|x_t - x_t'\| \leq 2\sqrt{T} \mathrm{diam}(\mathcal{X})$ and over $I_k, k \geq 0$ as $\textstyle\sum_{t \in I_k} \|x_t - x_t'\| \leq \textstyle\sum_{t \in S_k} \|x_t - x_{eq}\| + \textstyle\sum_{t \in S_k}  \|x'_t - x_{eq}\| \leq \textstyle \frac{3 - \alpha}{1 - \alpha} \mathrm{diam}(\mathcal{X}) +  \frac{2\overline{C} c}{\alpha} \sqrt{T} (\log T)^2$. Combining this with (\ref{eqn:dyn0}), we obtain $R_{N,T} \leq \textstyle 2L_{h,x}\sqrt{T} \mathrm{diam}(\mathcal{X}) + L_{h,x} K \big(\textstyle \frac{3 - \alpha}{1 - \alpha} \mathrm{diam}(\mathcal{X}) +  \frac{2\overline{C} \sqrt{c}}{\alpha} \sqrt{T} (\log T)^2\big) + (\ref{eqn:nstepregretbound})$, and it remains to bound $K$. Note $\mathbb{E} K \geq  c \log \frac{T+1}{2\sqrt{T} + 1}$, $\mathrm{Var}(K) \leq 2 \log T$, and Bernstein's inequality yields $\mathbb{P} (K \leq 2\mathbb{E}K) \geq 1 - \exp \Big(- \frac{c^2 (\log \frac{T}{2\sqrt{T}+1})^2}{(4+ \frac{2}{3}c^2) \log T}\Big)$. Bounding $K$ by $2\mathbb{E}K \leq 2c \log T$ gives the desired result. 
\end{proof}
This instantaneous bound implies asymptotic $N$-step dynamic regret of order $O(\sqrt{T} (\log T)^3)$ for Algorithm \ref{alg:nonmyopic}.
\section{Numerical Experiments}\label{sec:exper}
We conduct experiments using Python 3.7.4 and Anaconda on a laptop with 2.3 GHz 8-Core Intel Core i9 processor and 16GB DDR4 RAM. We use MOSEK \cite{mosek} for optimization.
We simulate an HVAC system (see Sect. \ref{subsec:hvac}), using a discrete time model from \cite{aswani2011reducing} with 15 minutes sampling interval and dynamics $x_{t+1} = k_r x_{t} - k_c u_{t} + k_v v_t + q_{t}$, where $x_t \in [20, 24]$ in $^\circ C$, $u_t \in [0, 0.5]$ is AC duty cycle, $v_t$ is outside temperature in $^\circ C$, and $q_t$ is heating load due to occupants. We assume $r_t = -c_t \sim \mathcal{N}\left(h(x_t, u_t, \theta_0), \sigma^2\right)$ for $h(x_t, u_t, \theta_0) = \gamma_1 p_t  u_t + (x_t - \gamma_2 - v_t)^2$ where $p_t$ is the electricity price assumed to follow a peak-pricing plan between 12-6 p.m. over an 24 hour day. The  $\gamma_1 p_t  u_t$ accounts for energy use, and $v_t + \gamma_2$ indicates a setpoint preference that adjusts with outside temperature \cite{ASHRAE}. We suppose $\theta_0 = \left[q_t, \gamma_1, \gamma_2\right]$ are unknown to the controller, and use $\sigma = 1, k_r = 0.64, k_c = 2.64, k_v = 0.10$ \cite{aswani2011reducing}. We assume $v_t$ and $q_t$ are generated from a sinusoidal distribution with a single peak over 24 hours and average values of $6.98$ and $17$, respectively.  All metrics are averaged across 1000 replicates. 
\begin{figure}
            \centering
    \adjustbox{trim=0 0.2cm 0 0.2cm}{%
            \includegraphics[width=0.75\textwidth]{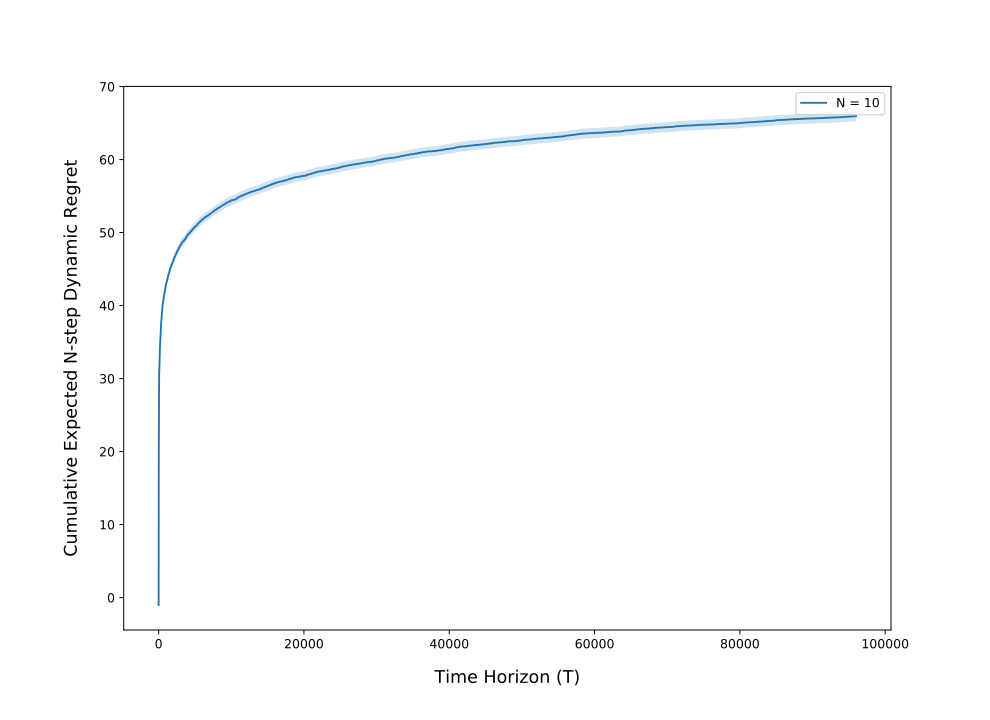}}
            \caption{Expected $10$-step dynamic regret. The shaded region represents the standard error over 1000 replications.}
            \label{fig:N10}
\end{figure}

Fig. \ref{fig:N10} shows regret up to time $T=100,000$ of the $N=10$ policy. These results are compatible with our asymptotic regret bound $O(\sqrt{T} (\log T)^3)$. Fig. \ref{fig:cost} compares cumulative expected costs of the $N=1$ and $N = 10$ policies by subtracting the expected cost of $\Lambda^{\epsilon, 10}_t(\mathcal{F}'_t)$ from that of $\Lambda^{\epsilon, 1}_t(\mathcal{F}'_t)$. Lower costs are obtained with $N=10$.
\begin{figure}
    \centering
    \adjustbox{trim=0 0.2cm 0 0.2cm}{%
\includegraphics[width=0.75\textwidth]{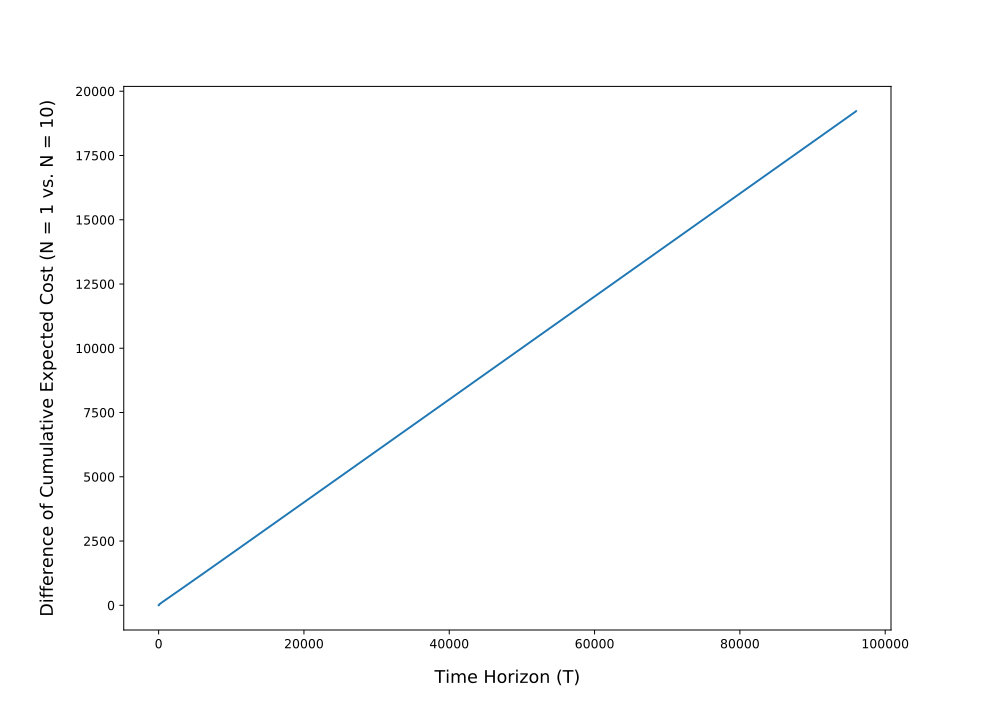}}
    \caption{Difference of cumulative expected costs of the policies for $N = 1$ and $N = 10$.}
    \label{fig:cost}
\end{figure}
\section{Conclusion}\label{conc}
This paper studies the intersection of nonlinear MPC and RL. Stability is one of the unique (and not previously well-studied) issues that arises with RL for nonlinear systems. We develop a new class of LBMPC policies that we prove achieves low regret, which is supported by our numerical experiments.
\bibliography{RegretLBMPC} 

\end{document}